\documentclass{article}
\usepackage{amssymb}
\usepackage{amsmath, amsthm}
\usepackage{graphicx}
\usepackage{tikz}
\newtheorem{theorem}{Theorem}[section]
\newtheorem{lemma}{Lemma}[section]

\newtheorem{definition}{Definition}[section]
\newtheorem{proposition}{Proposition}[section]
\newtheorem{remark}{Remark}[section]

\begin{document}
\title
{Linear Reduction and Homotopy Control for Steady  
Drift-Diffusion Systems in Narrow Convex Domains} 
%{Global Solution Curves for Boundary-Value
%Problems: Continuation as a Method for Existence and Approximation
\author{Joseph W. Jerome\footnotemark[1]}
\date{}
\maketitle
\pagestyle{myheadings}
\markright{}
\medskip

\vfill
{\parindent 0cm}


\bigskip
 {\bf 2020 AMS classification numbers:}35A01, 35B30, 35Q49, 46N20, 65J05 

\fnsymbol{footnote}
\footnotetext[1]{Department of Mathematics, Northwestern University,
                Evanston, IL 60208.} 
\begin{abstract}
This article develops and applies results, originally introduced in
earlier work, for the existence of homotopy curves, terminating
at a desired solution. 
We describe the principal hypotheses and results in section two;
right inverse approximation is at the core of the theory.
We apply this theory in section three  
to the basic drift-diffusion equations. 
The carrier densities are not
assumed to satisfy Boltzmann statistics and the Einstein relations are
not assumed. By proving the existence of the homotopy curve, we validate
the underlying  
computational 
framework of a predictor/corrector scheme, where the corrector utilizes an
approximate Newton method. 
The analysis depends on the assumption of 
domains of narrow width. However, no assumption is made regarding the
domain diameter.
\end{abstract}

\section{Introduction.}
\label{Introduction}
Potential-driven transport is now central in mathematical modeling in 
science and engineering. For electrostatic potentials, 
application areas include semiconductor transport \cite{Lundstrom}, 
ion channel transport \cite{Keener}, 
as well as the
behavior of electrochemical systems \cite{Rubin}. 
Various approaches have been employed to analyze these systems.
The reference \cite{semicon}
established existence by the creation of an implicit composite solution mapping,
essentially an abstract integral operator equation for systems. 
This has proven to be effective for analysis, but remains intricate,
particularly as regards approximation (cf.\ \cite{KerJer}). In the present 
article, we employ a
direct differentiable system map. This shifts the burden to one
of establishing regularity and derivative 
mapping surjectivity. Fixed point theory
is eliminated in favor of homotopy. In effect, existence is established if
the desired system can be represented as the terminus of a solution curve
which begins with a standard linear special case. 
The general result is found in section two, and is referenced from
\cite{Newton}.
The application to the drift-diffusion model is carried out in section
three. Although not explicitly considered, this result provides the
rigorous infrastructure for the 
Euler predictor-Newton corrector approximation method.
The homotopy curve emanates from the center of a ball of radius governed
by the problem data, and remains within this ball until the solution is
achieved.

We conclude the introduction with comments regarding the
analysis literature in this subject. The first rigorous study 
of the drift-diffusion semiconductor model appears
to be that of Mock \cite{Mock}. This was followed by \cite{Seidman}. A
framework for \cite{semicon} was contained in \cite{BJR}. An extensive
study appeared in \cite{Mark}. These studies are distinguished by the choice of
boundary conditions, domain geometry, physical parameters, and
recombination assumptions. 
For example, mixed boundary conditions, with polyhedral geometry, is apt
to reduce the expected solution regularity. 
%%Moreover, in rigorous modeling,
%mobilities must be electric field dependent to achieve saturation of current. 
The
reference \cite{Jerbook} includes these technical issues. 

The boundary conditions of the present
article are somewhat simplified for the semiconductor application.
However, they are realistic as applied to the Poisson-Nernst-Planck model
for ion channel transmission and electrochemical systems.   
For these applications, existence is less important than system response
to parameter variation and to analysis of specialized case studies. For
examples of the latter, see \cite{PJ, BCEJ}.

In summary, the model considered in this article is a basic scientific
model of drift-diffusion, without carrier recombination, and with 
convex bounded domains in Euclidean space of dimension $N =2,3$. As
formulated, the model is applicable to each of the three
application areas just referenced, when Dirichlet data are
appropriate. 

Finally, we observe that all of the analytical investigations 
cited above make use of
the Einstein relations, which are valid 
via arguments from statistical
mechanics. Although we do not assume either 
these relations, or Boltzmann
statistics,  
we do assume  
that the
physical region $\Omega$ has sufficiently small width, permitting an
effective application of the Poincar\'{e} inequality. 
Specifically, this
is used both in the $\lambda = \lambda_{0} = 0$ case and 
in the construction of inverses of the linearized equations. 
The width assumption does not restrict the domain diameter, 
either explicitly or implicitly. 
%It has not been established that the
%Einstein relations extend to such domains.   
The use of the basic variables 
is consistent with the study \cite{PNPNS}, where multi-physics
effects are included to supplement the Poisson-Nernst-Planck model.

\section{Prior Results}
We begin with the assumptions for an existence theory.
\subsection{Existence of a homotopy solution curve in Banach space}
The hypotheses stated  here are slightly stronger than those presented
in \cite{Newton}, in order to facilitate verifiability. We will make a more
precise statement at the conclusion of the following definition.
\begin{definition}
\label{hypotheses}
Suppose that Banach spaces $X$ and $Z$, a connected, open subset $U$ of
$X$, and a bounded, connected  
closed subset $\Lambda $ of ${\mathbb R}$ are given, together with
a mapping, 
\begin{equation}
F: {\bar U} \times \Lambda \mapsto Z.
\end{equation}  
%It is assumed that $F$ is continuously Fr\'{e}chet differentiable. In
%addition, 
The following are assumed.
\begin{enumerate}
\item
The partial derivatives $F_{w}$ and $F_{\mu}$ are 
Lipschitz continuous in the operator norm. 
%in the variables $w$ and $\mu$,
%resp. 
Specifically, there is a constant $C$ such that, globally,
\begin{eqnarray}
\|F_{w}(v, \lambda)  - F_{w}(w, \mu)\|_{B[X,Z]} 
&\leq& C  \{\|v - w\|_{X}^{2} + |\lambda - \mu|^{2}\}^{1/2}, \nonumber \\
\|F_{\mu}(v, \lambda)  - F_{\mu}(w, \mu)\|_{B[{\mathbb R},Z]}
&\leq& C  \{\|v - w\|_{X}^{2} + |\lambda - \mu|^{2}\}^{1/2}. 
\label{Lip}
\end{eqnarray} 
\item
$F_{w}(w, \mu)$ has an approximate inverse $H(w, \mu)$ 
satisfying, globally,
\begin{eqnarray}
\|F_{w}(w, \mu) H(w, \mu) z - z \|_{Z} &\leq& M \|F(w,\mu)\|_{Z}
\|z\|_{Z} \label{Inv},\\
\|H(w, \mu) z \|_{X} & \leq & M\|z \|_{Z}. 
\label{BInv}
\end{eqnarray}
\end{enumerate}
\end{definition}
\begin{remark}
The assumptions for $F_{w}$
and
$F_{\mu}$ imply the following remainder formula and estimate:
\begin{equation*}
%\label{QR}
F(w, \mu) =F(v, \lambda) +F_{w}(v, \lambda)(w - v) + F_{\mu}(v, \lambda) 
(\mu -\lambda) + R(v, w; \lambda, \mu), 
\end{equation*}
\begin{equation}
\label{QR}
\|R(v,w; \lambda, \mu)\|_{Z} \leq C [\|v -w\|_{X}^{2} + (\lambda -
\mu)^{2}].
\end{equation}
This follows from a direct estimate of the following representation
for the remainder.
\begin{equation*}
R(v,w; \lambda, \mu) = \int_{0}^{1} [F_{w}(v + s(w - v), \lambda)
-F_{w}(v, \lambda)] (w - v) \; ds
\;+
\end{equation*}
\begin{equation}
 \int_{0}^{1} [F_{\mu}(w, \lambda+t(\mu -\lambda))
- F_{\mu}(w, \lambda)] (\mu - \lambda) 
+ [F_{\mu}(w, \lambda)- F_{\mu}(v,\lambda)]  
(\mu - \lambda) \; dt.
\end{equation}
In \cite{Newton}, (\ref{QR}) is included in the hypotheses, whereas the
Lipschitz continuity of $F_{\mu}$ is not. Here, it is explicitly
included.
\end{remark}
We have the following.
\begin{theorem}
\label{exisgen}
Suppose that Banach spaces $X$ and $Z$, a closed interval $[\lambda_{0},
\lambda_{1}]$, a closed ball $B_{r} \subset X$, 
with interior $B_{r}^{o}$, 
and a mapping $F$ are given,
satisfying Definition \ref{hypotheses}. We suppose that 
(\ref{Lip}, \ref{Inv}, \ref{BInv})
of
Definition \ref{hypotheses} hold and also the following:

(i) $F(u_{0}, \lambda_{0}) = 0$, where $u_{0}$ is the center of $B_{r}$.

(ii) $r$ is sufficiently large relative to $M$ and $\lambda_{1} -
\lambda_{0}$:
\begin{equation}  
r \geq \max(C, M)(\lambda_{1} - \lambda_{0})\left[ 1 + 
\sup_{v, \lambda} \|F
_{\mu}(v, \lambda)\|_{B[{\mathbb R},Z]} \right].
\label{bigr}
\end{equation}
Then there exists a solution set $u = u(\lambda)$ such that 
\begin{equation}  
u(\lambda) \in B_{r}^{o}, \, F(u(\lambda), \lambda) = 0, \lambda \in
[\lambda_{0}, \lambda_{1}].
\end{equation}
\end{theorem}
\begin{remark}
This result follows from \cite[Th.\ 4.1]{Newton} and its proof. 
The proof is based upon a predictor/corrector method which ensures that
the iterates remain in the radius of convergence of Newton's method; 
here, this means the interior of $B_{r}$.
The argument uses the chain of inequalities given by (3.10) of \cite{Newton}. 
However, for the purposes here, two of the inequalities in the chain 
need to be replaced by strict inequalities:
that preceding 
the term $\rho^{-1}/4$ and the final inequality preceding the
term $3r/4$.
When modified in this way, each predictor is seen to lie in the
interior of $B_{r}$.

Note that the viability of inequality (\ref{bigr}) requires the decoupling
of the norm of $F_{\mu}$ from the choice of $r$. This is the most serious
challenge to the application of this theory to systems of partial
differential equations. 

\end{remark}
\subsection{Surjectivity of partial derivatives}
An essential hypothesis of the theory is a parameter dependent approximate
right inverse for the derivative map. Here, we provide a result which
guarantees, for the application considered, the existence of a right
inverse. The following result is a concatenation of 
\cite[Th.\ 27.A, Prop.\ 27.6]{ZeidlerII} 
\begin{proposition}
\label{surj}
Suppose that $X$ is a reflexive Banach space and that 
$L: X \mapsto X^{\ast}$ is a given operator with decomposition
$L = L_{1} + L_{2}$, satisfying the
following conditions.
\begin{enumerate}
\item
$L_{1}$ is monotone and continuous.
\item
$L_{2}$ is strongly continuous.
\item
$L$ is bounded.
% and pseudomonotone.
\item
$L$ is coercive.
\end{enumerate}
Then
$L$ is surjective.
\end{proposition}

\section{Existence of a Homotopy Curve for Drift-Diffusion Systems}
We apply the results of the preceding section to construct a solution
curve.
We begin by writing the system in strong form with parametric dependence. 
The dependent variables are electrostatic potential, charge density 
for negative carriers, and charge density for positive carriers,
denoted by $u, n, p$, resp. Carrier transport is not assumed to 
satisfy the Einstein relations. Boltzmann statistics are not assumed.
Moreover, additional carriers are easily accommodated if desired.
The system is given as follows. In the following section,
we will provide a precise definition of solution, which in this article is
defined as a strong solution for $u$ and a weak solution for $n$ and $p$.
\begin{eqnarray}
-\nabla^{2} u &=& \lambda (p - n) + D, \nonumber \\
-\nabla \cdotp ( d_{n}\nabla n - c_{n}n \nabla u) 
&= 0,& \nonumber \\ 
-\nabla \cdotp (d_{p}\nabla p + c_{p}p \nabla u) 
&= 0.&  
\label{sys9}
\end{eqnarray}
Here, $d_{n}, c_{n}, d_{p}, c_{p}$ are positive constants, related to 
diffusion and carrier drift, resp. In the first equation, $D$ represents
so-called permanent charge, 
and is an arbitrary $L^{\infty}$ function. 
Also, units in this equation are chosen so that
the ratio of charge modulus to dielectric constant is one.
Dirichlet boundary conditions are imposed as
follows, via the boundary trace operator $\Gamma$, for given functions 
$a_{u}, a_{n}, a_{p}$.
\begin{equation}
\Gamma u = 
\Gamma a_{u},
\Gamma  n = \Gamma a_{n},
\Gamma  p = \Gamma a_{p}.
\end{equation}
The domain $\Omega \subset {\mathbb R}^{N}$ is an open
convex bounded set. 
Here, $N = 2,3$.
The parameter $\lambda$
ranges between $0, 1$. 
\subsection{The fundamental mapping $F$}
We begin with the component definitions of $F$. 
\begin{definition}
\label{defF}
Set ${\mathcal H} = (H^{2}(\Omega) \cap H^{1}_{0}(\Omega)) \times
H^{1}_{0} (\Omega) \times H^{1}_{0}(\Omega)$ and  
${\mathcal G} = L^{2}(\Omega) \times H^{-1}(\Omega) \times H^{-1}(\Omega)$. 
Arbitrary elements of ${\mathcal H}$ are designated as 
$(\rho, \sigma, \tau)$. 
Standard norms are employed, except for the equivalent norm in 
$H^{1}_{0}$: $\|f\|_{H^{1}_{0}} = \|\nabla f\|_{L^{2}}$. 
%${\mathcal B}_{R}$ denotes the open ball
%in ${\mathcal H}$, centered at the origin, of radius $R$.  
Then 
$F: {\mathcal H} \times [0, 1] \mapsto 
{\mathcal G}$
 is defined via its components as follows.
Here, $(u, n, p) := (\rho, \sigma, \tau) + (a_{u}, a_{n}, a_{p})$ for a
fixed specified function $(a_{u}, a_{n}, a_{p}) \in 
H^{2} \times H^{1} \times H^{1}$.
\begin{eqnarray}
F_{1}(\rho, \sigma, \tau; \lambda) &=& 
-\nabla^{2} u + \lambda (n - p) - D, \nonumber \\
F_{2}(\rho, \sigma, \tau; \lambda) &=& 
-\nabla \cdotp ( d_{n}\nabla n - c_{n}n \nabla u),
\nonumber \\ 
F_{3}(\rho, \sigma, \tau; \lambda) &=& 
-\nabla \cdotp (d_{p}\nabla p + c_{p}p \nabla u).  
\label{sys11}
\end{eqnarray}

$F_{1}$ has natural range in $L^{2}$.
The precise (standard) meaning of the action of 
$F_{2}, F_{3}$ is given by the
following. In describing this action, we write an
arbitrary vector test function  for 
the final two components of ${\cal G}$ as $(\phi, \psi)$.
We have:
\begin{eqnarray}
%F_{1}(\rho, \sigma, \tau; \lambda)  &=& 
%\int_{\Omega} [\nabla u \cdotp \nabla \omega  + \lambda (n - p)
%\omega] \;dx, \nonumber \\
F_{2}(\rho, \sigma, \tau; \lambda)[\phi]  &=& 
\int_{\Omega}[( d_{n}\nabla n - c_{n}n \nabla u)\cdotp \nabla \phi
] \; dx, \nonumber \\ 
F_{3}(\rho, \sigma, \tau; \lambda)[\psi] &=& 
\int_{\Omega}[( d_{p}\nabla n + c_{p}n \nabla u)\cdotp \nabla \psi
] \; dx. 
\end{eqnarray}
By a homotopy solution curve is meant a set 
$\{h_{\lambda} =(\rho_{\lambda}, \sigma_{\lambda}, \tau_{\lambda})\} 
\subset {\mathcal H}, 0 \leq \lambda \leq 1$, such that
$F(h_{\lambda}, \lambda) = 0, \; \forall \lambda$.
\end{definition}
\begin{remark}
As noted earlier, we will place restrictions on the width of the convex
domain $\Omega$. We use the terminology of \cite[p.\ 359]{Leoni}. A domain
$\Omega$ has finite width $d$ if the latter is the smallest positive
number such that $\Omega$ lies between two parallel planes separated by a
distance $d$.
\end{remark}
\begin{lemma}
For each fixed $\lambda \in [0, 1]$, the mapping $F$ is well-defined from 
${\mathcal H}$ to ${\mathcal G}$.
\end{lemma}
\begin{proof}
The analysis for $F_{1}$ is clear. We consider $F_{2}, F_{3}$.
Each integral in the definition has a product integrand which is $L^{1}$,
as seen by the H\"{o}lder inequality. Moreover, the corresponding estimate
includes a factor dominated  by the $H^{1}_{0}$ norm of the test function. 
Note that in Euclidean dimension not
exceeding four, an $H^{1}$ function is in $L^{4}$ \cite{Adams}. 
This permits the
conclusion.
\end{proof}
\subsection{Solution at $\lambda = 0$}
We analyze now the $\lambda = 0$ case. The system decouples into three
independent linear equations. We have the following. 
\begin{proposition}
There is a number $d_{0}$ such that for domain width $d < d_{0}$, the
system for $\lambda = 0$ has a solution. More precisely, 
there is an element $h_{0} \in {\mathcal H}$ satisfying
$F(h_{0},0) = 0$.
Moreover, we have the explicit formula,
\begin{equation}
d_{0}^{2} = 4\frac{1}{\|D\|_{L^{\infty}}
+\|a_{u}\|_{L^{\infty}}} \min\left(\frac{d_{n}}{c_{n}}, 
\frac{d_{p}}{c_{p}}\right). 
\label{d0}
\end{equation}
Finally, the solution is unique if $d < d_{0}$.
\end{proposition}
\begin{proof}
We notice that (\ref{d0}) makes use of the fact that, for $N \leq 3$, 
$H^{2}$ functions are bounded.
Now the first equation of (\ref{sys9}) is a Poisson equation:
\begin{equation}
-\nabla^{2} \rho = D + \nabla^{2} a_{u}, 
\end{equation}
subject to homogeneous boundary values. The results of \cite{VA} imply the
existence of a solution $\rho \in H^{2}$. Standard methods imply
uniqueness.

The second equation of (\ref{sys9}) can be rewritten as
\begin{equation}
-\nabla \cdotp ( d_{n}\nabla \sigma - c_{n}\sigma \nabla u)
=
\nabla \cdotp ( d_{n}\nabla a_{n} - c_{n}a_{n} \nabla u)
\end{equation}
The rhs of this equation is identified with a member of $H^{-1}$.
It follows that the equation is satisfied if we can prove surjectivity of
the mapping defined by the lhs.
To achieve this, we use Proposition \ref{surj}, with 
\begin{equation}
L_{1} = -\nabla \cdotp  d_{n}\nabla \sigma, 
L_{2} = \nabla \cdotp  c_{n}\sigma \nabla u.
\label{surj2}
\end{equation} 
The norm properties imply that $L_{1}$ is monotone and continuous.
We next consider the operator $L_{2}$. 
 Suppose that $\sigma_{k}$ converges weakly to $\sigma$ in
$H^{1}_{0}$. By the Sobolev compact  embedding theorem, 
$\sigma_{k}$ converges  to $\sigma$ in $L^{p}, 1 \leq p < 6$.
We claim that the vector sequence $ \sigma_{k} \nabla u$ converges to
$\sigma \nabla u$ in $L^{2}$. 
This is demonstrated directly by the Schwarz inequality. Once established,
we use the following characterization of the $H^{-1}$ norm:
\begin{equation}
\|L_{2}(\sigma_{k} - \sigma) \|_{H^{-1}} = 
\sup_{\|\phi\|_{H^{1}_{0}}\leq 1} \int_{\Omega} \nabla u (\sigma_{k} - \sigma)
\cdotp \nabla \phi \; dx. 
\end{equation}
If the Schwarz inequality is applied to this expression, then it follows
from the preceding that the sequence 
$L_{2}{\sigma_{k}}$ converges to $L_{2}\sigma$.
Strong continuity follows.
%The analysis for the third component is similar.

The third condition, coerciveness of $L$, follows directly from the following
inequality \cite[p.\ 501]{ZeidlerII}:
\begin{equation}
\label{coe}
\langle L \sigma, \sigma \rangle \geq
C(\sigma, \sigma)_{H^{1}_{0}}  
%-C_{1} \|\sigma\|_{H^{1}_{0}}^{2}
%-C_{2}  \|\Delta u \|_{L^{2}}^{2}.
\end{equation}
To derive (\ref{coe}), 
we reason as follows. 
\begin{eqnarray}
\langle L \sigma, \sigma \rangle 
&=& 
d_{n}\|\sigma\|_{H^{1}_{0}}^{2} + c_{n}\frac{1}{2}\int_{\Omega} \nabla u \nabla
\sigma^{2} \; dx  
\nonumber \\
&\geq & 
d_{n}\|\sigma\|_{H^{1}_{0}}^{2} 
-\frac{c_{n}}{2} \|D+\nabla^{2} a_{u}\|_{L^{\infty}}  
\|\sigma\|_{L^{2}}^{2} \nonumber \\
&\geq&
d_{n}\|\sigma\|_{H^{1}_{0}}^{2} 
-(d^{2}/2) \frac{c_{n}}{2} \|D+\nabla^{2} a_{u}\|_{L^{\infty}}  
\|\nabla \sigma\|_{L^{2}}^{2}, 
\end{eqnarray}
where we have used the Poincar\'{e} inequality.
Inequality (\ref{coe}) now follows for any $d < d_{0}$. 
%An identical argument applies to the third componen.

Finally, we establish the boundedness of $L$. Suppose that ${\mathcal B}$
is bounded in $H^{1}_{0}$, with $\| \sigma \|_{H^{1}_{0}} \leq \beta$.
We obtain an upper bound for the norm of $L \sigma$ in $H^{-1}$.
The first term satisfies 
$ \|L_{1} \sigma\|_{H^{-1}} \leq d_n \beta$.
The estimation of $ \|L_{2} \sigma\|_{H^{-1}}$ involves two terms,
resulting from the application of the divergence operator.
\begin{itemize}
\item
The term involving $\nabla^{2} u$, which is an $L^{\infty}$ function, 
 is estimated as 
\begin{equation}
c_{n} \{\sup_{\|\phi\|_{H^{1}_{0}} \leq 1} \int_{\Omega} \nabla^{2}u \sigma
\phi \; dx \}
\leq c_{n} \|\nabla^{2} u\|_{L^{\infty}} \|\sigma\|_{L^{2}}
\|\phi\|_{L^{2}}.
\end{equation}
An application of the Poincar\'{e} inequality to each of the terms
involving $\sigma$ and $\phi$ 
completes the analysis of this term.
\item
The term including $\nabla \sigma$ follows a similar pattern. We begin with 
\begin{equation}
c_{n}\{\sup_{\|\phi\|_{H^{1}_{0}} \leq 1} \int_{\Omega} \nabla u 
\cdotp \nabla \sigma
\phi \; dx \}
\leq c_{n} \|\nabla u\|_{L^{4}} \|\nabla \sigma\|_{L^{2}}
\|\phi \|_{L^{4}},
\end{equation}
followed by Sobolev's  inequality. We obtain a similar upper bound for
this term.
\end{itemize}
The term involving the positive ion current is handled identically.
This concludes the proof of existence. 

To establish uniqueness, suppose that two solutions $n_{1}, n_{2}$ exist
for the second equation of the system, and set $\sigma = n_{1} - n_{2}$.  
By the above arguments, (\ref{coe}) is shown to hold for $d < d_{0}$, so
that $\sigma = 0$. A similar argument holds for the third component.
\end{proof}
\subsection{Fr\'{e}chet differentiability}
In this section, we establish, in the sense of  Fr\'{e}chet 
differentiability, the existence of the partial 
derivatives $F_{h}$ and $F_{\lambda}$. Moreover, $F_{h}$ and $F_{\lambda}$
are continuous so that $F^{\prime}$ exists and is continuous.
We employ the notation $h = (\rho, \sigma, \tau)$ for 
an arbitrary element of ${\mathcal H}$ at which the Fr\'{e}chet derivative
is computed. 
The variables acted upon by $F_{h}$ are designated by 
$h^{\prime} = (\rho^{\prime}, \sigma^{\prime}, \tau^{\prime})$.
In the process, 
we use some general results, documented in
\cite{ZeidlerI}. The structure of the mapping $F$ is such that 
one  may begin with the G\^{a}teaux derivative. Because of directional
uniformity, it will be seen that these partial derivatives are Fr\'{e}chet
derivatives \cite[Prop.\ 4.8]{ZeidlerI}. Further, it will be shown that the 
partial derivatives are
continuous in the joint variable $(h, \lambda)$, so that $F^{\prime}(h,
\lambda)$
is continuously Fr\'{e}chet differentiable \cite[Prop.\ 4.14]{ZeidlerI}.

\begin{lemma}
\label{Fh}
The partial derivative $F_{h}$ exists as a 
Fr\'{e}chet derivative 
which is continuous in the variable $(h, \lambda)$. 
Its evaluation, as a linear operator applied to 
the variables $\rho^{\prime}, \sigma^{\prime}, \tau^{\prime}$, 
is 
given as follows:  
\begin{eqnarray}
&F_{1, h}(\rho^{\prime}, \sigma^{\prime}\, \tau^{\prime})  = 
-\nabla^{2} \rho^{\prime}  + 
\lambda (\sigma^{\prime} - \tau^{\prime}),
\nonumber \\
&F_{2, h}(\rho^{\prime}, \sigma^{\prime}, \tau^{\prime})[\phi]  = 
\int_{\Omega}[( d_{n}\nabla \sigma^{\prime} - c_{n}\sigma^{\prime} \nabla u
- c_{n}n \nabla \rho^{\prime})\cdotp \nabla \phi]
 \; dx, \nonumber \\
&F_{3, h}(\rho^{\prime}, \sigma^{\prime}, \tau^{\prime})[\psi]  = 
\int_{\Omega}[( d_{p}\nabla \tau^{\prime} +c_{p} \tau^{\prime} \nabla u  
+ c_{p}p \nabla \rho^{\prime})\cdotp \nabla \psi
] \; dx.
\label{der1}
\end{eqnarray}
As previously,  $(u, n, p) := (\rho, \sigma, \tau) + (a_{u}, a_{n},
a_{p})$.
\end{lemma}
\begin{proof}
The most efficient way to verify the validity of the system 
(\ref{der1}) is to compute the
G\^{a}teaux derivatives with respect to the components of $h$; then, 
verify sufficient conditions
for Fr\'{e}chet differentiability. 
Finally, the
partial derivatives with respect to the elements of $h$ can be summed to
obtain (\ref{der1}) because of the continuity of the associated component
derivatives. 

The G\^{a}teaux partial derivatives are directional derivatives in the
directions $\rho^{\prime}, \sigma^{\prime}$, and $\tau^{\prime}$. 
Because of the linear structure of the
variables, the difference quotients reduce to the derivatives
straightforwardly. Summation yields (\ref{der1}).  
Continuity of $F_{h}$ is a direct consequence of the linearity 
of the variables associated with $(h, \lambda)$. 
Indeed, the duality norm, coupled with H\"{o}lder's inequality, yields the
result.
\end{proof}

There is a parallel result for $F_{\lambda}$. The proof is similar
to the preceding lemma so that we omit the proof. 
\begin{lemma}
\label{Flam}
The partial derivative $F_{\lambda}$ exists as a 
Fr\'{e}chet derivative 
which is continuous in the variable $(h, \lambda)$. 
Its evaluation, as a linear operator applied to the real variable
$\theta$,  
is given as follows:  
\begin{eqnarray}
F_{1, \lambda}(\theta)  &=& 
 \theta (n  - p),
 \nonumber \\
F_{2, \lambda}(\theta)  &=& 
0 \nonumber \\ 
F_{3, \lambda}(\theta)  &=& 
0. \label{Fdersys}
\end{eqnarray}
\end{lemma}
%\begin{proof}
%The Fr\'{e}chet differentiability follows the previous proof. The
%continuity requires an analysis of the quadratic term $np$. Because of the
%Sobolev embedding theorem, convergence in $H^{1}$ implies convergence in
%$L^{4}$, so that the product converges in $L^{2}$. This implies
%continuity via estimation of the operator duality norm..
%\end{proof}
\subsection{Lipschitz continuity for derivatives}
We begin with the statement of Lipschitz continuity. 
%The Fr\'{e}chet differentiability properties of the previous section were
%seen to be local properties. On the other hand, Lipschitz properties are
%global. As a result, it will be necessary to impose bounds for the
%elements $h$ in this section.
\begin{proposition}
\label{Lipw}
%For $h$ in a bounded subset of ${\mathcal H}$, 
The derivative mapping $F_{h}$ is globally 
Lipschitz continuous in the variable $(h, \lambda)$. 
%Specifically,  
%there is a constant $M$, depending only on physical and domain dependent
%constants, such that
%for arbitrary elements $h_{j} = (\rho_{j}, \sigma_{j}, \tau_{j}) 
%\in {\mathcal H}$, $j = 1,2$, 
%\begin{equation}
%\|(F_{h}(h_{2}, \lambda) - 
%F_{h}(h_{1}, \lambda)) h^{\prime} \|_{{\mathcal G}} \leq   
%M \|h_{2} - h_{1} \|_{{\mathcal H}} \|h^{\prime} \|_{{\mathcal H}}.
%\end{equation}
\end{proposition}
\begin{proof}
It suffices to show that $F_{h}(\cdotp, \lambda)$ is Lipschitz continuous for each fixed 
$\lambda$ and that $F_{h}(h, \cdotp)$ Lipschitz continuous for each fixed $h$. 

We begin with the first case.
Because of the structure of the system (\ref{der1}), only differences in
the second and third components are relevant. These are quite similar in
their analysis, so that we will consider differences in $F_{2, h}$. 
The following estimates hold for $\|\phi\|_{H^{1}_{0}} \leq 1$.
They are required for $H^{-1}$ estimation.
\begin{eqnarray}
c_{n}\|\sigma^{\prime}\nabla (\rho_{2} - \rho_{1}) \cdotp \nabla \phi\|_{L^{1}} 
& \leq & c_{n}\|\sigma^{\prime}\|_{L^{4}} \| \rho_{2} -  \rho_{1}\|_{W^{1,4}} 
 \leq c_{n}C \|\sigma^{\prime}\|_{H^{1}_{0}} \|\rho_{2} - \rho_{1}\|_{H^{2}} 
\nonumber \\
c_{n}\|(\sigma_{2}-\sigma_{1})\nabla \rho^{\prime}\cdotp \nabla \phi\|_{L^{1}} 
& \leq & c_{n}\|\nabla \rho^{\prime}\|_{L^{4}}\|\sigma_{2}-\sigma_{1}\|_{L^{4}} 
 \leq c_{n}C \|\nabla \rho^{\prime}\|_{H^{1}_{0}}
\|\sigma_{2}-\sigma_{1}\|_{H^{1}_{0}} 
\nonumber 
\end{eqnarray}
In these inequalities, $C$ is a generic constant obtained from the product
of Sobolev embedding constants.
As mentioned, the third component analysis is exactly parallel. 
Since the duality norm for ${\mathcal G}$ involves the supremum over
test functions with ${\mathcal H}$ norm not exceeding one,  
we may obtain an upper bound for  
\begin{equation}
\|(F_{h}(h_{2}, \lambda) - 
F_{h}(h_{1}, \lambda)) h^{\prime} \|_{{\mathcal G}}
\end{equation}
from the estimates above. Indeed, each estimate is bounded above by a
constant times 
$ \|h_{2} - h_{1} \|_{{\mathcal H}} \|h^{\prime} \|_{{\mathcal H}}$.
The result that $F_{h}(\cdotp, \lambda)$ is Lipschitz continuous  
follows from an application of basic $\ell^{2}$ inequalities. 

The second case, with fixed $h$,  
is immediate.
This concludes the proof.
\end{proof}
\begin{proposition}
\label{Lipmu}
The partial derivative $F_{\lambda}(h, \lambda)$ is 
globally Lipschitz continuous. 
\end{proposition}
\begin{proof}
The partial derivative $F_{\lambda}(h, \cdotp)$  
is invariant for each
fixed $h$. It is sufficient then to consider $F_{\lambda}(\cdotp, \lambda)$
for each fixed $\lambda$. In this case, 
the linearity of the components of $h$ implies the result.
%It follows that if $h$ is taken from a bounded subset of
%${\mathcal H}$, then $F_{\lambda}$ satisfies a Lipschitz condition    
%Lipschitz
%continuous with constant $M$. 
\end{proof}
\begin{remark}
The hypothesis of (\ref{Lip}) has now been verified. The constant $C$ is obtained directly from the current analysis. 
\end{remark}
\subsection{A bounded inverse for $F_{h}$}
Recall that $h_{0}$ denotes the solution of the system for $\lambda = 0$.
In the previous subsections, the action of $F(\cdotp, \lambda)$ has included 
all of
${\mathcal H}$. Here, we restrict the domain of $F(\cdotp, \lambda)$ to a ball  
$S= \{h \in {\mathcal H}: \|h\|_{\mathcal H} \leq R\}$, where 
$R = \|h_{0}\|_{{\mathcal H}} + r$, with $r$ to be specified. Observe  that
the general theory guarantees that the homotopy curve lies in an open ball
of radius $r$, centered at $h_{0}$. 
We note here that the restriction of $F$ to $S$ guarantees the logical
consistency of the choice of the domain width $d$,
so that it remains
bounded away from $0$. This enters decisively into the contraction
mapping theorem discussed below.
We now construct inverses for $F_{w}$ along the curve.

For each fixed $\lambda$, 
and each fixed $h$, we will show that there is a bounded linear
operator $H(h, \lambda)$ 
from ${\mathcal G}$ to ${\mathcal H}$ which is an
inverse for $F_{h} (h, \lambda)$. 
\begin{proposition}
There is a positive real number $d_{1} = d_{1}(R)$ such that, if
the width $d$ of $\Omega$ satisfies $d \leq  d_{1}$ and $d < d_{0}$, 
for each pair $(h, \lambda)$, the linear operator $F_{h}(h, \lambda)$ is
bijective. More precisely, for each $g \in {\mathcal G}$, there is a
unique element
$h^{\prime} \in {\mathcal H}$ such that 
$F_{h}(h, \lambda)h^{\prime} = g$.
In particular, 
$F_{h}(h, \lambda)$ is invertible. The norm $M$   
of the inverse may be selected so that $M = 2$. 
\end{proposition}
\begin{proof}
There are two stages to the proof.
\begin{description}
\item
(a)
Let $g \in {\mathcal G}$ be given, and consider the system 
$F_{h}(h, \lambda) h^{\prime} = g$. 
We show that this system has a unique solution $h^{\prime}$ under
assumptions on the width $d$.
\item
(b)
There is a uniform  bound  for 
the norms of the inverses of (a). This is achieved by an appropriate
uniform selection of $d(R)$.
\end{description}
We consider these in turn.

(a)
%We use Proposition \ref{surj}. For the decomposition, we define $L_{1}$ as
Although the system is linear, we employ the 
contraction mapping theorem
for an appropriately defined operator ${\mathcal T}$. 
For the remainder of the proof, we use the equivalent norm on 
${\mathcal H}$ given by the Laplacian norm for the first component and,
as previously mentioned,  the 
derivative part of the $H^{1}_{0}$ norm for the remaining components.
Also, for convenience of notation, we write
$\Delta = \nabla^{2}$.
In order to reformulate the system for an application of the contraction  
mapping theorem, we define $-\Delta^{-1}:{\mathcal G} \mapsto {\mathcal H}$
by the componentwise definition:
\begin{equation}
-\Delta^{-1} g_{i} = f_{i},  
\end{equation}
where 
\begin{equation}
(\nabla f_{i}, \nabla \phi)_{L^{2}} = g_{i}(\phi), i=1,2,3.
\end{equation}
Results of \cite{VA} for convex domains 
allow us to conclude that the Laplacian defines an isomorphism
between
$H^{2} \cap H^{1}_{0}$ and $L^{2}$. 

The surjectivity will follow from a fixed point argument applied to an
equivalent system, given by
\begin{equation}
h^{\prime} = {\mathcal T}h^{\prime} - \Delta^{-1}{\tilde g}.
\end{equation} 
Here, ${\tilde g}$ is obtained from $g$ by by dividing the second and
third components by $d_{n}, d_{p}$, respectively. The same operations are
applied to the components of $F_{h}$ in the formation of ${\mathcal T}$.

Since ${\mathcal T}$ is a strict contraction if and only if ${\mathcal T}
- \Delta^{-1} {\tilde g} := {\mathcal S}$ is, we consider ${\mathcal T}$. 
For reference, we write the
form explicitly.
\begin{eqnarray}
&{\mathcal T}_{1}(\rho^{\prime}, \sigma^{\prime}\, \tau^{\prime})  = 
\lambda \Delta^{-1}(\sigma^{\prime} - \tau^{\prime}),
\nonumber \\
&{\mathcal T}_{2}(\rho^{\prime}, \sigma^{\prime}, \tau^{\prime}) = 
  \frac{c_{n}}{d_{n}}\Delta^{-1} \nabla \cdotp (
\sigma^{\prime} \nabla u
+ n \nabla \rho^{\prime})
 \nonumber \\
&{\mathcal T}_{3}(\rho^{\prime}, \sigma^{\prime}, \tau^{\prime}) = 
 - \frac{c_{p}}{d_{p}}\Delta^{-1} \nabla \cdotp (
\tau^{\prime} \nabla u
+ p \nabla \rho^{\prime})
\end{eqnarray}
The proof will show that ${\mathcal T}$ maps ${\mathcal H}$ into itself.
This is inherited by ${\mathcal S}$.
We verify the strict contractive property. Since ${\mathcal T}$ is linear,
it suffices to estimate $\|{\mathcal T} h^{\prime}\|_{\mathcal H}$.
A fundamental tool in this estimation is the use of the version of the
Poincar\'{e} inequality in \cite[Th.\ 12.17]{Leoni}. 
This inequality permits $L^{p}$ estimation in terms of a 
$W^{1,p}_{0}$ estimation with a constant proportional to the finite width of
$\Omega$. 
Recall that this property was used effectively to derive the solution
$h_{0}$. 
Now, for sufficiently small width, the contraction constant is less
than $1$. We present the argument as follows, with emphasis on the case $N
= 3$, the most delicate case. The case $N = 2$ is implied by the
arguments. 
\begin{itemize}
\item
$H^{2} \cap H^{1}_{0}$ estimation of ${\mathcal T}_{1}$
\end{itemize}
\begin{equation}
\|\Delta^{-1}\sigma^{\prime}\|_{H^{2}} \leq \|\Delta^{-1}\|  
\|\sigma^{\prime}\|_{L^{2}}
\leq (d/\sqrt{2})\|\Delta^{-1} \| \|\nabla \sigma^{\prime}\|_{L^{2}}.  
\end{equation}
A similar result holds for the term involving $\tau^{\prime}$.
\begin{itemize}
\item
$H^{1}_{0}$ estimation of ${\mathcal T}_{2}$
\end{itemize}
When the divergence operator is applied, each of the four terms is of the
form $-\Delta^{-1} \chi$, where it will be shown that $\chi \in L^{3/2}$.
In order to estimate the $H^{1}_{0}$ norm of $-\Delta^{-1} \chi$, we
observe that, by definition, this term is numerically equal to
the $H^{-1}$ norm of $\chi$. 
A representation of this norm on $H^{-1}$  
will be chosen, allowing for a direct application of the
Poincar\'{e} inequality. This form was used in \cite{JR}, in the study of
the Stefan problem and is used for $L^{p}$ functions which define
continuous linear functionals, provided $-\Delta^{-1} \chi \in H^{1}_{0}$,
as is the case for $\chi \in L^{3/2}$. 
In particular, for $\chi$ described above, we have
\begin{equation}
\label{dualequiv}
\|\chi\|_{H^{-1}} =\sqrt{\int_{\Omega} {\mathcal R}\chi \; \chi \; dx}.
\end{equation} 
Here, we have written ${\mathcal R}$ for the map $-\Delta^{-1}$.

For $\chi$ in (\ref{dualequiv}), we have the following inequalities.
\begin{eqnarray}
\int_{\Omega} {\mathcal R}\chi \; \chi \; dx &\leq& \|{\mathcal R} \;\chi
\|_{L^{3}} \|\chi \|_{L^{3/2}}\\
&\leq& (d/3^{1/3}) \|\nabla {\mathcal R} \chi \|_{L^{3}}
\|\chi\|_{L^{3/2}}\\
&\leq& C (d/3^{1/3}) \|{\mathcal R} \chi \|{W^{2, 3/2}}
\|\chi\|_{L^{3/2}}\\
&\leq& C (d/3^{1/3})\|{\mathcal R}\|  
 \|\chi\|_{L^{3/2}}^{2}.
\end{eqnarray}
The first inequality follows from H\"{o}lder's inequality; the second from
the Poincar\'{e} inequality\cite{Leoni}; the third from Sobolev's
embedding theorem \cite{Adams}; and the forth from \cite{VA}. 

In this estimation, we will choose $\chi = \chi_{1} \chi_{2}$, where
$\chi_{1} \in L^{2}$ and $\chi_{2} \in H^{1}$. 
Typically, these functions are vector-valued.
H\"{o}lder's inequality
implies that $\chi \in L^{3/2}$.
Indeed, when the inequality is applied with $p = 4/3$ and $q=4$ to 
$\|\chi_{1} \chi_{2} \|_{L^{3/2}}$, one obtains the respective product of
the $L^{2}$ and $L^{6}$ norms; Sobolev's inequality gives the desired
result:
\begin{equation}
\|\chi\|_{H^{-1}}^{2} \leq 
C^{\prime} d \|\chi_{1} \|^{2}_{L^{2}} \|\chi_{2}\|^{2}_{H^{1}_{0}}.  
\end{equation}
Here, $C^{\prime}$ depends only on operator and embedding constants.
These general remarks will now be applied to the terms above obtained by the
application of the divergence operator as follows. Consider ${\mathcal
T_{2}}$. 
\begin{equation}
  \frac{c_{n}}{d_{n}} (
\nabla \sigma^{\prime}\cdotp  \nabla u
 + \sigma^{\prime} \nabla^{2} u
+ \nabla n \cdotp \nabla \rho^{\prime}
+ n \nabla^{2} \rho^{\prime}).
\end{equation}
Each of these four terms is an example of the generic function $\chi =
\chi_{1} \chi_{2}$ discussed previously. In each case, the terms involving
$u$ and $n$  and their derivatives have norm bounds depending on a 
constant $C(R)$. A similar
analysis holds for the third component of
${\mathcal T}$. In summary, by standard  
inequalities, we obtain an inequality of the form:
\begin{equation}
\|{\mathcal T} h^{\prime}\|_{\mathcal H} \leq d C(R)
\|h^{\prime}\|_{\mathcal H}.
\end{equation}

We thus obtain a contraction if $d$ is sufficiently small.
In particular, we may choose $d \leq d_{2}$ so that the
contraction constant does not exceed $1/2$. Note that this choice of $d$
depends only upon $R$.
Thus, ${\mathcal S}$ has a unique fixed point; the existence of
an inverse follows.  

(b) To obtain a uniform estimate for the inverse norm, we begin with the
fixed point relation,
\begin{equation}
h^{\prime} = {\mathcal T} h^{\prime} + \Delta^{-1} {\tilde g},
\end{equation}
and we estimate the ${\mathcal H}$ norm.
We obtain,
\begin{equation}
\|h^{\prime}\|_{\mathcal H} \leq  \|{\mathcal T} h^{\prime}\|_{\mathcal H} 
+ \|{\tilde g}\|_{\mathcal G}.
\end{equation}
Since we have selected $d$ to ensure an upper bound of $1/2$ for the
contraction constant, we conclude that $2$ serves as an upper bound for
the inverse mapping. This is the assertion $M = 2$ in the statement of the
proposition.
\end{proof}
\subsection{The existence theorem}
We begin with an essential lemma.
\begin{lemma}
A uniform  estimate for $F_{\lambda}$ is given as follows.
\begin{equation}
\|F_{\lambda}\| \leq \|a_{n}\|_{L^{2}} 
+\|a_{p}\|_{L^{2}} 
+ (2d/\sqrt{2}) R. 
\end{equation}
Here, $d$ is the width of $\Omega$ and $R = r + \|h_{0}\|_{{\mathcal H}}$.
Also, $h_{0}$ is a solution 
of the system for $\lambda = 0$ and $r$, to be specified below, is
the radius of a ball centered at $h_{0}$.
\end{lemma}
\begin{proof}
We proceed 
from the explicit representation of $\|F_{\lambda}\|$, and calculate
the operator norm. 
It is bounded above by the $L^{2}$ norm. 
The factor $2d/\sqrt{2}$ follows from the Poincar\'{e}
inequality applied to the elements $\sigma, \tau$.
\end{proof}
\begin{remark}
One sees that the choices, for $0 < \alpha_{0} \leq 1$, 
\begin{eqnarray*}
d &=& \frac{\alpha_{0}}{2\sqrt{2} \max(C, 2)}, \\ 
r &=& \frac{1}{1-\alpha_{0}/2}\max(C, 2)(1 
+\|a_{n}\|_{L^{2}} 
+\|a_{p}\|_{L^{2}})
+ \frac{\alpha_{0}}{2 - \alpha_{0}}\|h_{0}\|_{{\mathcal H}},  
\end{eqnarray*}
lead to the satisfaction of the requirement for the existence of the
homotopy curve, with the proper choice of 
$\alpha_{0}$. This is made specific as follows.
\end{remark}
The following theorem is now a direct consequence of Theorem
\ref{exisgen}.
\begin{theorem}
Consider the drift-diffusion system $F(u, \lambda) = 0$, where $F$ is
defined in Definition \ref{defF}. 
If the width $d$ of $\Omega$
satisfies $d < d_{0}$, where $d_{0}$ is defined in  
(\ref{d0}), then 
there is a solution to the system for $\lambda = 0$. If $r$
is defined as in the remark above, and $d$ also satisfies $d \leq d_{1}, 
 d \leq d_{2}$, 
then there exists a homotopy solution
curve within the open ball of radius $r$, centered at $h_{0}$. 
\end{theorem}
\subsection{Nonnegativity for ion densities}
Thus far, we have studied the system independently of the sign properties
associated with the ion densities. In the process, we have not made
assumptions on the signs of $a_{n}, a_{p}$.
We now assume their nonnegative boundary traces. Note that charge voids
are not eliminated, so that a situation can arise which is excluded when
Boltzmann statistics are assumed. 
\begin{theorem}
Suppose a homotopy solution curve $h_{\lambda} = 
(\rho_{\lambda}, \sigma_{\lambda}, \tau_{\lambda})$, 
exists as described in Definition \ref{defF}.
We write 
\begin{equation*}
(u_{\lambda}, n_{\lambda}, p_{\lambda}) = (a_{u}, a_{n}, a_{p}) +
h_{\lambda}
\end{equation*}
for the basic variables. 
There is a number $D_{0}$ such that, if the width $d$ satisfies $d <
D_{0}$, and 
if $a_{n} \geq 0, a_{p} \geq 0$ on $\partial \Omega$,
then $n_{\lambda} \geq 0, p_{\lambda} \geq 0$ 
on $\Omega$ 
for each fixed $\lambda$.
\end{theorem}
\begin{proof}
We observe that $\|h_{\lambda}\|_{{\mathcal H}} \leq R, \forall \lambda$.
Consider the system for a fixed $\lambda, 0\leq \lambda \leq 1$.
For clarity, we suppress the subscript  $\lambda$
for the variables. We consider the
second equation of the system, and we employ $n^{-}$ as a test function,
where $n^{-}$ is the negative part of $n$. Here, $n^{-}$ satisfies
$n = n^{+} + n^{-}$. The assumption that $a_{n} \geq 0$ on 
$\partial \omega$
implies that the boundary trace of $n^{-}$ is zero, so that 
$n^{-} \in H^{1}_{0}$. 
We now use the second equation to show that $n^{-}$ is
the zero function. 

We have:
\begin{equation}
0 = d_{n} \int_{\Omega} \nabla n \cdotp \nabla n^{-} \; dx 
- c_{n}\int_{\Omega} n \nabla u  \cdotp \nabla n^{-} \; dx.  
\end{equation}
It follows that 
\begin{eqnarray*}
0 &=& 
d_{n} \int_{\Omega} |\nabla n^{-}|^{2} \; dx 
- c_{n}/2 \int_{\Omega}  \nabla u  \cdotp \nabla (n^{-})^{2} \; dx  
\\
&=& 
d_{n} \int_{\Omega} |\nabla n^{-}|^{2} \; dx 
+ c_{n}/2 \int_{\Omega}  \nabla^{2} u  (n^{-})^{2} \; dx \\
&\geq& 
d_{n} \|n^{-}\|^{2}_{H^{1}_{0}} 
- c_{n}/2 \|\nabla^{2} u\|_{L^{4}}\; \|n^{-}\|_{L^{4}}\;
\|n^{-}\|_{L^{2}} \\
&\geq&  
d_{n} \|n^{-}\|^{2}_{H^{1}_{0}} - c_{n}/2 \;d \; C(D, R) \|\nabla
n^{-}\|^{2}_{L^{2}}. 
\end{eqnarray*}
The constant $C(D, R)$ arises from the estimation of the $L^{4}$ norm of
$\nabla^{2} u$ 
as determined from (\ref{sys9}), and from Sobolev embedding constants.
It is now immediate that, if $d$ is sufficiently small, then 
\begin{equation}
\|n^{-}\|_{H^{1}_{0}} = 0,
\end{equation}
which implies that $n^{-} = 0$.
A similar analysis yields the result that $p^{-} = 0$ if $d$ is
sufficiently small.
The number $D_{0}$ in the statement of the theorem is chosen as the
largest number such that the final expression  above, as well as the
corresponding expression  for $p^{-}$, is nonnegative.
\end{proof}
%\begin{thebibliography}{10}
\subsection{Final remarks}
We have analyzed a basic scientific model of electrodiffusion and drift in
an alternative way, through the basic variables. Although we avoid the
Einstein relations and Boltzmann statistics, our analysis is restricted to
narrow convex domains. We establish the existence of a homotopy solution
curve, starting from the uniquely defined center of a specifically located 
sphere in
function space. 
The result then allows the application of 
\cite[Th.\ 3.1]{Newton}, which describes a predictor/corrector algorithm,
terminating at the starting iterate of a convergent Newton sequence.
The conditions of the general theory, verified here for the
drift-diffusion model, allow the implementation of the algorithm.
Although the results are theoretical, they provide a basic strategy for
the solution of the system.

As a concluding observation, we note that the arguments are not sign
sensitive, so that the electrostatic potential can be replaced by
potentials which induce drift. 

\end{document}